\documentclass[11pt]{article}

\usepackage[T1]{fontenc}
\usepackage{lmodern}
\usepackage{microtype}
\usepackage{amsmath,amssymb,amsthm}
\usepackage[numbers,sort&compress]{natbib}
\usepackage[margin=1in]{geometry}
\usepackage[hidelinks]{hyperref}

\allowdisplaybreaks
\numberwithin{equation}{section}

\theoremstyle{plain}
\newtheorem{theorem}{Theorem}[section]
\newtheorem{lemma}[theorem]{Lemma}
\newtheorem{proposition}[theorem]{Proposition}
\newtheorem{corollary}[theorem]{Corollary}
\newtheorem{problem}[theorem]{Problem}

\theoremstyle{definition}
\newtheorem{definition}[theorem]{Definition}

\theoremstyle{remark}

\newcommand{\Cay}{\operatorname{Cay}}

\title{Cyclic Hamilton Cycle Decompositions of Carousel Tournaments of Order
\texorpdfstring{$pq$}{pq}}

\author{Hongci Liao\thanks{School of Science,
	Xihua University,
	Chengdu 610039, P. R. China; \texttt{hcliao@xhu.edu.cn}}
	\and
	Yongju Peng\thanks{School of Mathematics,
	Guangxi University,
	Nanning 530004, P. R. China; \texttt{3566909596@qq.com}}
	\and
	Guang Li\thanks{School of Mathematics and Statistics,
	Shandong University of Technology,
	Zibo 255100, P. R. China; \texttt{lig@sdut.edu.cn}}
	\and
	Yingbin Ma\thanks{Corresponding author: School of Mathematics and Statistics, Henan Normal University,
	Xinxiang 453007, P.R. China; \texttt{mayingbincw@htu.cn}}}
\date{}

\begin{document}

\maketitle

\begin{abstract}
	Kelly's conjecture asks whether every regular tournament admits a Hamilton
	cycle decomposition. Motivated by its symmetry-preserving extension, we study
	cyclic Hamilton decompositions of carousel tournaments. For an odd integer $n$, let
	\[
	T_n=\Cay\left(
	\mathbb Z_n,\left\{1,2,\ldots,\frac{n-1}{2}\right\}
	\right)
	\]
	be the carousel tournament. We ask whether $T_n$ has a Hamilton decomposition
	invariant under translation by $1$. Although the answer is immediate when $n$
	is prime, composite orders introduce a genuine obstruction: nonunit differences
	generate short cycles rather than Hamilton cycles. We resolve a general
	composite-order family by proving that, whenever $n=pq$ for primes
	$7\le p<q$, the tournament $T_n$ admits a cyclic Hamilton cycle decomposition.
	The proof combines Hamiltonian difference sequences over prime fields with a
	matching argument that constructs two base paths with disjoint difference
	sets in $\mathbb Z_{pq}$.
	Thus our result gives an infinite family supporting the cyclic, symmetry-preserving
	extension of the Hamilton decomposition problem for regular tournaments.
	
\end{abstract}


\section{Introduction}\label{intro}

A Hamilton decomposition of a graph or digraph is a partition of its edge set into
Hamilton cycles. The problem is one of the oldest decomposition problems in graph
theory. Walecki proved that the complete graph of odd order has a Hamilton
decomposition; see~\cite{alspachWonderfulWalecki2008,alspachDecompositionCycles1990}.
In the directed setting, Tillson~\cite{tillsonHamiltonianDecomposition1980} showed
that the complete symmetric digraph admits a Hamilton decomposition in every order
other than $4$ and $6$.

A tournament is an orientation of a complete graph. A Hamilton decomposition
of a tournament can exist only when the tournament is regular, since every Hamilton
cycle contributes one incoming and one outgoing arc at each vertex. The converse is
the content of the following classical conjecture.

\medskip
\noindent\textbf{Kelly's conjecture.}
\emph{Every regular tournament has a Hamilton cycle decomposition.}
\medskip

Kelly's conjecture has played a central role in the study of Hamilton decompositions
of directed graphs; see also the survey~\cite{kuhnSurveyHamilton2012}.
K\"uhn, Osthus and Treglown~\cite{kuhnHamiltonDecompositionsRegular2010} proved
an approximate version, and K\"uhn and Osthus~\cite{kuhnHamiltonDecompositionsExpanders2013}
subsequently proved the conjecture for all sufficiently large tournaments. Their
result is considerably more general: sufficiently large regular robust outexpanders
of linear degree have Hamilton decompositions.

Hamilton decomposition problems have also been studied for Cayley graphs.
For undirected Cayley graphs, Hamilton decompositions are known in several
important cases, including degree-four Cayley
graphs~\cite{bermondHamiltonianDecomposition1989}, Cayley graphs associated with
minimal generating sets of abelian groups of odd
order~\cite{liuHamiltonianDecompositionsOdd1996}, and Paley
graphs~\cite{alspachPaleyGraphs2012}. Cayley tournaments form a natural class of
regular, vertex-transitive tournaments because every translation is an
automorphism.

The classical definition of a Hamilton decomposition is purely existential: it asks
whether the edge or arc set can be partitioned into Hamilton cycles, but it does not
record how the automorphisms of the graph act on the resulting partition. It is
therefore natural to extend the concept by requiring the decomposition itself to
preserve a prescribed part of the symmetry of the graph. Let $\Gamma$ be a regular
graph or digraph and let $H\le\operatorname{Aut}(\Gamma)$. We call a Hamilton
decomposition $\mathcal C$ of $\Gamma$ \emph{$H$-invariant} if
\[
C^h\in\mathcal C
\qquad\text{for every }C\in\mathcal C\text{ and }h\in H.
\]
Equivalently, $H$ acts on the set of Hamilton cycles of the decomposition.
Ordinary Hamilton decomposition theory corresponds to making no prescribed symmetry
requirement, while the $H$-invariant version studies the compatibility between a
decomposition and an algebraic action on the underlying graph. This leads to a
general extension of the classical problem.

\begin{problem}\label{prob:invariant}
	Given a regular graph or digraph $\Gamma$ and a subgroup
	$H\le\operatorname{Aut}(\Gamma)$, determine whether $\Gamma$ admits an
	$H$-invariant Hamilton decomposition.
\end{problem}

Even when a Hamilton decomposition of $\Gamma$ is known to exist,
Problem~\ref{prob:invariant} need not be automatic: the cycles must be organized
into complete $H$-orbits, and their stabilizers constrain the possible orbit
lengths. The problem thus retains algebraic information that is invisible in an
arbitrary decomposition. It also provides a common setting for Hamilton
decomposition problems in Cayley objects and for cyclic decompositions in design
theory.

Cyclic decompositions are usually constructed by specifying a small collection of
base cycles or base paths and developing them under translation. Their arc sets are
encoded by differences. This method goes back to cyclic cycle systems of complete
graphs and underlies the work of Buratti and Del
Fra~\cite{burattiExistenceCyclic2003,burattiCyclicHamiltonian2004}, the extension
to complete multigraphs and cocktail-party graphs~\cite{burattiCyclicLambda2010},
and the corresponding results for complete multipartite
graphs~\cite{merolaCyclicSymmetric2017,akinolaCyclicSymmetricHamilton2025}.
Jordon and Morris studied cyclic Hamilton systems of the complete graph minus a
$1$-factor and directed cyclic Hamilton systems of complete symmetric
digraphs~\cite{jordonCyclicHamiltonian2008,jordonMorrisDirectedCyclic2009}. In an
undirected graph or a symmetric digraph, both differences $d$ and $-d$ are
available. A tournament contains exactly one of them. Consequently, a difference
construction for a tournament must control not only orbit lengths and vertex
repetitions, but also the orientation of every selected arc.

For an odd integer $n$, consider the carousel tournament
\[
T_n=\Cay(\mathbb Z_n,S_n),
\qquad
S_n=\left\{1,2,\ldots,\frac{n-1}{2}\right\}.
\]
Thus the vertices of $T_n$ are the residues modulo $n$, and $x$ sends an arc
to $x+d$ for each $d\in S_n$. The tournament $T_n$ is regular, and translation
by $1$ acts regularly on its vertices. Let
\[
H_n=\{x\mapsto x+t:t\in\mathbb Z_n\}.
\]
Then $H_n$ is a regular cyclic subgroup of $\operatorname{Aut}(T_n)$. An
$H_n$-invariant Hamilton decomposition is traditionally called a \emph{cyclic
	Hamilton cycle decomposition}. Thus carousel tournaments provide a canonical and
arithmetically structured instance of Problem~\ref{prob:invariant}. The
corresponding specialization is the following.

\begin{problem}\label{prob:carousel}
	For which odd integers $n$ does the carousel tournament $T_n$ admit a cyclic
	Hamilton cycle decomposition?
\end{problem}

There is a sharp distinction between prime and composite orders. If $n$ is prime,
every $d\in S_n$ is a unit of $\mathbb Z_n$, and
\[
(0,d,2d,\ldots,(n-1)d)
\]
is a Hamilton cycle. As $d$ ranges over $S_n$, these cycles partition the arc
set of $T_n$ and are individually invariant under cyclic development. Thus
Problem~\ref{prob:carousel} is immediate for prime $n$. If $n$ is composite, a
nonunit difference $d$ has additive order strictly smaller than $n$ and
therefore generates several short cycles rather than a Hamilton cycle. Combining
these short orbits into Hamilton cycles while preserving cyclicity is the essential
difficulty.

We initiate the composite-order theory of cyclic Hamilton decompositions of
carousel tournaments by treating a general arithmetic family, namely
$n=pq$ for two distinct odd primes. Write
\[
T_{p,q}=T_{pq}=\Cay\left(
\mathbb Z_{pq},
\left\{1,2,\ldots,\frac{pq-1}{2}\right\}
\right).
\]
The nonunit differences in its connection set are precisely the positive multiples
of $p$ and $q$ below $pq/2$. A base path of length $p$ cannot contain a
multiple of $p$, while a base path of length $q$ cannot contain a multiple of
$q$. Hence the multiples of $q$ must be assigned to a base path of length $p$,
and the multiples of $p$ to one of length $q$. All remaining differences are
units and may be handled individually. The problem is therefore reduced to
constructing two base paths whose difference sets are disjoint, with prescribed
nonunit differences and prescribed endpoint divisibility.

Our main result establishes an infinite composite-order family in this
equivariant setting and provides a positive answer to Problem~\ref{prob:carousel}
for all products $pq$ with $7\le p<q$.

\begin{theorem}\label{thm:main}
	Let $p<q$ be primes with $p\ge7$. Then $T_{p,q}$ has a cyclic Hamilton cycle
	decomposition.
\end{theorem}

The proof has two main ingredients. First, we construct an explicit Hamiltonian
difference sequence over each of the prime fields $\mathbb Z_p$ and
$\mathbb Z_q$. Its partial sums visit every field element, while its difference
multiplicities have a controlled form: one nonzero residue is omitted, two occur
twice, and all remaining nonzero residues occur once. The omitted class is chosen
away from the residues occupied by the required nonunit differences. Second, a
bipartite matching lifts the prescribed multiplicities modulo $p$ to actual
differences in $S_{pq}$ having distinct residues modulo $q$. An exchange of one
matching edge enforces the endpoint condition. These distinct residues leave enough
choices to construct the second base path class by class, avoiding the first, and a
final replacement enforces its endpoint condition. The restriction $p\ge7$ is
used only in the residue-class estimates that guarantee these replacements.

The paper is organized as follows. Section~\ref{sec:preliminaries} introduces
base paths and proves the cyclic development lemma, reducing the problem to a
partition of the connection set into sufficient difference sets.
Section~\ref{sec:long-difference-sets} constructs the required prime-field
difference sequences and lifts them to disjoint sufficient difference sets of
types $p$ and $q$. Section~\ref{sec:main-proof} then completes the proof of
Theorem~\ref{thm:main}.

\section{Preliminaries}\label{sec:preliminaries}

Throughout this paper, $\mathbb Z_n$ denotes the cyclic group of integers
modulo $n$, written additively, and is identified with the set
$\{0,1,\dots,n-1\}$. All arithmetic in $\mathbb Z_n$ is understood modulo $n$.
For each $v\in\mathbb Z$, let $[v]_n$ denote the element in $\mathbb Z_n$
corresponding to $v$ modulo $n$. We write
$\mathbb Z_n^*=\mathbb Z_n\setminus\{0\}$ for the nonzero elements and
$\mathbb Z_n^\times$ for the units of $\mathbb Z_n$.

For $t\in\mathbb Z_n$ and any subset, path, cycle, or subgraph $X$ on vertices in $\mathbb Z_n$, let $X+t$ denote the object obtained by adding $t$ to each of its vertices. This operation
is called \emph{translation by $t$}; in particular, an arc $(u,v)$ is translated to $(u+t,v+t)$. For a family $\mathcal C$ of cycles, define $\mathcal C+t=\{C+t:C\in\mathcal C\}.$ A Hamilton cycle decomposition $\mathcal C$ is called \emph{cyclic} if $ C+1\in\mathcal C$ for every $C\in\mathcal{C}.$

Let $n$ be odd and let $S\subseteq\mathbb Z_n\setminus\{0\}$ satisfy
\[
S\cap(-S)=\emptyset,
\qquad
S\cup(-S)=\mathbb Z_n\setminus\{0\}.
\]
The \emph{Cayley tournament} $\Cay(\mathbb Z_n,S)$ has
vertex set $\mathbb Z_n$, with an arc $(x,y)$ if and only if
$y-x\in S$. The element $y-x$ is called the \emph{difference} of the
arc $(x,y)$. Since translations preserve arc differences, every map
$x\mapsto x+t$ is an automorphism of
$\Cay(\mathbb Z_n,S)$.

In this paper, we consider $T_{p,q}=\Cay\left(\mathbb Z_{pq},\left\{1,2,\ldots,\frac{pq-1}{2}\right\}\right),$ where $7\leq p<q$ are distinct odd primes.

\subsection{Directed base path and cyclic developments}\label{base-path}

We now introduce the local construction used to generate Hamilton cycles by
translation. Let $d$ be a positive integer, and let $P$ be a directed path of
length $d$ with arcs $(v_0,v_1),(v_1,v_2),\dots,(v_{d-1},v_d)$, where
$v_0,v_1,\dots,v_d\in\mathbb Z_n$ are distinct. We write
$P=(v_0,v_1,\ldots,v_d)$. Let the multiset $\Delta(P)$ be defined by
\[
\Delta(P)=\{\delta_1,\delta_2,\dots,\delta_d\},
\qquad
\delta_i=v_i-v_{i-1}\quad (1\leq i\leq d),
\]
where $\delta_i$ is the directed difference of the arc $(v_{i-1},v_i)$.
We say that $\Delta(P)$ is \emph{simple} if it contains no repeated elements.

Let $T=\Cay(\mathbb Z_n,S)$ be a Cayley tournament. Let
$P=(v_0,v_1,\ldots,v_d)$ be a directed path in $T$ with $v_0=0$, and suppose
that $d$ divides $n$. Then $P$ is a \emph{directed base path of length $d$} if
the following conditions hold:
	\begin{enumerate}
		\item[(1)] $\{[v_1]_d, [v_2]_d,\dots,[v_{d-1}]_d\}=\mathbb Z_d^*;$
		
		\item[(2)] $\Delta(P)\subseteq S$ and is simple;
		
		\item[(3)] $\gcd(v_d,n)=d.$
	\end{enumerate}

The next lemma is the main development principle underlying our construction.

\begin{lemma}\label{bpd}
	Let $P=(v_0,v_1,\dots,v_d)$ be a directed base path of length $d$, and put $h=v_d$. Then the paths
	\[
	P,\ P+h,\ P+2h,\ldots,
	P+\left(\frac nd-1\right)h
	\]
	form a directed Hamilton cycle, denoted by $C(P)$.
	Moreover, $\mathcal{C}(P)=\{C(P)+t:0\le t<d\}$ consists of $d$ Hamilton cycles and covers precisely all arcs whose differences belong to $\Delta(P)$, each exactly once. In particular,
	$\mathcal{C}(P)$ is invariant under translation by $1$.
\end{lemma}

\begin{proof}
	Since $\gcd(h,n)=d$, the additive order of $h$ in $\mathbb Z_n$ is
	$n/d$, and $\langle h\rangle=d\mathbb Z_n$. The terminal vertex of $P+jh$
	is the initial vertex of $P+(j+1)h$, while the final path returns to $0$.
	Consequently, the paths concatenate to form a directed cycle.
	
	The vertices appearing as initial vertices of its arcs are $v_i+jh,$ with $0\le i<d, 0\le j<\frac nd.$ Suppose that $v_i+jh=v_{i'}+j'h.$ Since $[v_0]_d,[v_1]_d, [v_2]_d,\dots,[v_{d-1}]_d$ are pairwise distinct, it follows that $i=i'$. As $h$ has order $n/d$, we conclude that $j=j'$.
	Thus the displayed vertices are pairwise distinct and therefore $C(P)$ is a directed Hamilton cycle.
	
	For a fixed $\delta_i\in\Delta(P)$, the arcs of difference $\delta_i$ in $C(P)$ have their initial vertices in the coset $v_{i-1}+\langle h\rangle=v_{i-1}+d\mathbb Z_n.$ Because the differences $\delta_1,\ldots,\delta_d$ are distinct, any translation fixing $C(P)$ must preserve this coset. Hence the translation stabilizer of $C(P)$ is exactly $d\mathbb Z_n$, and the translation orbit of $C(P)$ has size $d$. Its distinct members are
	therefore $\{C(P),C(P)+1,\ldots,C(P)+(d-1)\}.$
	
	As $t$ ranges over $0,1,\ldots,d-1$, the cosets $v_{i-1}+t+d\mathbb Z_n$ partition $\mathbb Z_n$. It follows that $\mathcal{C}(P)$ contains exactly one arc $(v, v+\delta_i)$ for every $v\in\mathbb Z_n$. This holds independently for every $\delta_i\in\Delta(P)$, proving the covering assertion. Since $\mathcal{C}(P)$ is a full translation orbit, it is invariant under translation by $1$.

\end{proof}

A set $D\subseteq S$ is called a \emph{sufficient difference set of type $r$} if there exists a directed base path $P$ of length $r$ such that $\Delta(P)=D$.
We denote such a set by $D_r$.
The next lemma expresses the terminal vertex condition of a directed base path in terms of the sum of its differences.

\begin{lemma}\label{es}
	Let $r$ be a divisor of $n$, and let $D_r$ be a sufficient difference set of type $r$.
	Then
	\[
	\gcd\left(\sum_{\delta\in D_r}\delta,n\right)=r.
	\]
	Equivalently, the element $\sum_{\delta\in D_r}\delta$ has additive order $n/r$ in $\mathbb Z_n$.
	In particular, if $s$ is a prime divisor of $n$ such that $s\nmid r$, then
	\[
	\sum_{\delta\in D_r}\delta\not\equiv0\pmod s.
	\]
\end{lemma}

\begin{proof}
	Let $P=(v_0,v_1,\ldots,v_r)$ be a directed base path realizing $D_r$, and write
	$\delta_i=v_i-v_{i-1}$ for $1\le i\le r$.
	Since $v_0=0$, we have
	\[
	\sum_{\delta\in D_r}\delta
	=
	\sum_{i=1}^{r}(v_i-v_{i-1})
	=
	v_r.
	\]
	The definition of a directed base path gives $\gcd(v_r,n)=r,$ which proves the first assertion. Hence the additive order of $v_r$ in $\mathbb Z_n$ is $n/r$.
	
	Now let $s$ be a prime divisor of $n$ with $s\nmid r$.
	If $\sum_{\delta\in D_r}\delta\equiv0\pmod s,$ then $s$ would divide both $n$ and
	$\sum_{\delta\in D_r}\delta$, and therefore
	\[
	s\mid
	\gcd\left(\sum_{\delta\in D_r}\delta,n\right)
	=
	r,
	\]
	contrary to $s\nmid r$.

\end{proof}

In particular, for distinct primes $p$ and $q$, sufficient difference sets of types $p$ and $q$ must
satisfy, respectively,
\[
\sum_{d\in D_p}d\equiv0\pmod p,
\qquad
\sum_{d\in D_p}d\not\equiv0\pmod q,
\]
and
\[
\sum_{d\in D_q}d\equiv0\pmod q,
\qquad
\sum_{d\in D_q}d\not\equiv0\pmod p.
\]
The nonzero congruence is essential. If the sum vanished modulo both
$p$ and $q$, then the terminal vertex would equal the initial vertex,
and translation by the terminal vertex could not join the translated
copies of the path into a Hamilton cycle.

A collection $\mathcal D$ of sufficient difference sets is called a
\emph{complete system of sufficient difference sets} for $S$ if
$\mathcal D$ is a partition of $S$.
The relevance of a complete system of sufficient difference sets is given by the following direct consequence of Lemma~\ref{bpd}.

\begin{corollary}\label{css}
	Let $\mathcal D=\{D_{r_1},\ldots,D_{r_k}\}$ be a complete system of sufficient difference sets for $S$. For each $i\in\{1,\ldots,k\}$, let $P_{r_i}$ be a directed base path realizing $D_{r_i}$, so that $
	\Delta(P_{r_i})=D_{r_i}.$ Then $\bigcup_{i=1}^k\mathcal{C}(P_{r_i})$ is a cyclic Hamilton cycle decomposition of $\Cay(\mathbb Z_n,S)$.
\end{corollary}

\subsection{The arithmetic structure of the connection set} \label{ds}

We now specialize the preceding framework to $n=pq$ and $S=\left\{1,2,\ldots,\displaystyle\frac{pq-1}{2}\right\}$, where $7\le p<q$ are distinct odd primes. The nonunit differences in $S$ are precisely the multiples of $p$ and $q$. Define
\[
M_p
=
\left\{p,2p,\ldots,\frac{q-1}{2}p\right\}\quad \text{and} \quad M_q
=
\left\{q,2q,\ldots,\frac{p-1}{2}q\right\}.
\]
Then $|M_p|=\displaystyle\frac{q-1}{2}, |M_q|=\displaystyle\frac{p-1}{2},$
and $M_p\cap M_q=\emptyset$. Hence $S=M_p\cup M_q\cup\bigl(S\cap\mathbb Z_{pq}^{\times}\bigr).$

For every $u\in S\cap\mathbb Z_{pq}^{\times}$, the path $P_u=(0,u)$ is a directed base path of length $1$. Thus every unit difference can be realized individually, and only the differences in
$M_p\cup M_q$ need to be incorporated into directed base paths of greater length.

The following lemma shows that $D_r$ contains no multiple of $r$.

\begin{lemma}\label{fp}
	Let $r>1$ divide $n$, then no element of $D_r$ is divisible by $r$. Consequently, suppose that $n=pq$ and that a complete system of sufficient difference sets consists of $D_p$, $D_q$ and sufficient singleton sets. Then $M_p\subseteq D_q$ and $M_q\subseteq D_p.$
\end{lemma}

\begin{proof}
	Let $P=(v_0,v_1,\ldots,v_r)$ be a directed base path realizing $D_r$, and write $\delta_i=v_i-v_{i-1},$ where $1\le i\le r.$
	For $1\le i<r$, the residues $[v_{i-1}]_r$ and $[v_i]_r$ are distinct, and hence $\delta_i\not\equiv0\pmod r.$
	Moreover, $\gcd(v_r,n)=r$ implies that $v_r\equiv0\pmod r,$ whereas $v_{r-1}\not\equiv0\pmod r$. Thus $\delta_r\not\equiv0\pmod r.$
	It follows that no difference in $D_r$ is divisible by $r$.
	
	In particular, $D_p$ contains no multiple of $p$, while $D_q$ contains no multiple of $q$. Furthermore, a difference can be realized by a directed base path of length $1$ only if it is a unit of $\mathbb Z_{pq}$.
	Therefore, the differences in $M_p$ can occur neither in $D_p$ nor in a sufficient singleton set and must belong to $D_q$.
	Similarly, every difference in $M_q$ must belong to $D_p$.

\end{proof}

By Lemma~\ref{fp}, the differences in $M_p$ must lie in $D_q$,
whereas those in $M_q$ must lie in $D_p$. It therefore suffices to
construct disjoint sufficient difference sets $D_p,D_q\subseteq S$
such that $M_q\subseteq D_p,$ $|D_p|=p$ and $M_p\subseteq D_q,$ $|D_q|=q.$

For any such pair, every element of $R=S\setminus(D_p\cup D_q)$ is a unit.
Since each singleton $\{u\}$ with $u\in R$ is a sufficient difference set, the family
\[
\{D_p,D_q\}\cup\bigl\{\{u\}:u\in R\bigr\}
\]
is a complete system of sufficient difference sets for $S$.

The main step in constructing $D_p$ and $D_q$ is to arrange the elements of each set into the difference sequence of a directed base path.
We first study the corresponding ordering problem modulo a prime using Hamiltonian difference sequences.

\subsection{A Hamiltonian difference sequence over prime fields}

Encoding a directed Hamilton cycle by its successive differences is
standard and is closely related to rainbow-difference cycles and
directed rotational terraces; see, for example, \citet{levSumsDifferencesHamiltonian2010}.
We use this encoding here for a more specific purpose: we need to control exactly how often
each nonzero residue occurs as a difference.

\begin{definition}
	Let $t$ be an odd prime. A sequence $\boldsymbol{\rho}=(\rho_1,\ldots,\rho_t)$
	of elements of $\mathbb Z_t^*$ is called a	\emph{Hamiltonian difference sequence} if its partial sums
	\[s_0=0,
	\qquad
	s_i=\sum_{j=1}^{i}\rho_j
	\quad (1\le i\le t)
	\]
	satisfy $s_t=0$ and $s_0,s_1,\ldots,s_{t-1}$ are pairwise distinct.
\end{definition}

Since $\mathbb Z_t$ has $t$ elements, the partial sums $s_0,s_1,\ldots,s_{t-1}$ then visit every element of $\mathbb Z_t$ exactly once before returning to $0$.
Equivalently, $\boldsymbol{\rho}$ is the sequence of successive differences of a directed Hamilton cycle on $\mathbb Z_t$.

The following lemma gives an explicit Hamiltonian difference sequence.

\begin{lemma}\label{pfp}
	Let $t\ge5$ be a prime, let $g$ be a primitive root modulo $t$, and choose $y\in\mathbb Z_t^*$. Define $z=-gy$ and $x=y+z=(1-g)y$ in $\mathbb Z_t$. Then
	\[
	\boldsymbol{\rho}_t(y,g)
	=
	\bigl(y,z,xg,xg^2,\ldots,xg^{t-2}\bigr)
	\]
	is a Hamiltonian difference sequence on $\mathbb Z_t$. Moreover, $x,y,z$ are pairwise distinct.
\end{lemma}

\begin{proof}
	Since $g$ has order $\varphi(t)=t-1\ge4$, we have $g\notin\{1,-1\}$.
	Hence $x=(1-g)y$ and $z=-gy$ are both nonzero.
	Moreover, $y=z$ would imply $g=-1$, while $x=y$ would imply $z=0$, and $x=z$ would imply $y=0$.
	Thus $x,y,z$ are pairwise distinct.
	
	Consider first the geometric difference sequence
	\[
	x,xg,xg^2,\ldots,xg^{t-2}.
	\]
	Let $u_0=0$ and $u_j=\sum_{i=0}^{j-1}xg^i=x\frac{g^j-1}{g-1}$ with $1\le j\le t-1.$
	
	If $0\le i<j\le t-2$ and $u_i=u_j$, then $g^i=g^j$, contradicting the fact that $g$ has order $t-1$.
	Therefore, $u_0,u_1,\ldots,u_{t-2}$ are pairwise distinct.
	Furthermore, $u_{t-1}=x\frac{g^{t-1}-1}{g-1}=0.$
	
	Since $g^0,g^1,\ldots,g^{t-2}$ run through $\mathbb Z_t^*$, the vertices $u_0,u_1,\ldots,u_{t-2}$ comprise $\mathbb Z_t\setminus\left\{-\frac{x}{g-1}\right\}.$
	Using $x=(1-g)y$, we obtain $-\frac{x}{g-1}=y.$
	Thus $y$ is the unique vertex not visited by the geometric sequence.
	
	The first difference $x$ corresponds to the arc from $0$ to $x$.
	Since $x=y+z$, this arc can be replaced by $0\longrightarrow y\longrightarrow x,$
	whose successive differences are $y$ and $z$.
	This inserts the missing vertex $y$, so the resulting sequence $\boldsymbol{\rho}_t(y,g)$ is a Hamiltonian difference sequence on $\mathbb Z_t$.

\end{proof}

Under the hypotheses and notation of Lemma~\ref{pfp}, we write $\boldsymbol{\rho}_t(y,g)=(\rho_1,\ldots,\rho_t)$ and for each $a\in\mathbb Z_t^*$, define
$\mu_a^{(t)}=\left|\left\{i\in\{1,\ldots,t\}:\rho_i=a\right\}\right|.$
The resulting multiplicities are
\[
\mu_a^{(t)}=
\begin{cases}
	0, & a=x,\\
	2, & a\in\{y,z\},\\
	1, & a\in\mathbb Z_t^*\setminus\{x,y,z\}.
\end{cases}
\]
In particular, $\sum_{a\in\mathbb Z_t^*}\mu_a^{(t)}=t$ and $\sum_{a\in\mathbb Z_t^*}\mu_a^{(t)}a=0$ in $\mathbb Z_t.$

For fixed $g$, the map $y\mapsto(1-g)y$ is a bijection of $\mathbb Z_t^*$, so the zero-multiplicity class $x$ can be chosen arbitrarily.
In the next step, we choose $x$ so that no required nonunit difference is congruent to $x$ modulo $t$.
In the subsequent lifting, $\mu_a^{(t)}$ specifies how many differences congruent to $a$ modulo $t$ are selected, while $\rho_1,\ldots,\rho_t$ specify the order in which they occur along the directed base path.

\section{Construction of the sufficient difference sets}
\label{sec:long-difference-sets}

Throughout this section, let $7\le p<q$ be distinct odd primes and put
$n=pq$. We continue to write
$T_{p,q}=\Cay(\mathbb Z_{pq},S)$, where
$S=\{1,2,\ldots,(pq-1)/2\}$.

By the reduction in Section~\ref{ds}, it remains to construct disjoint
sufficient difference sets $D_p$ and $D_q$ of types $p$ and $q$,
respectively. We record this task in the following proposition.

\begin{proposition}
	\label{prop:long-difference-sets}
	There exist disjoint sufficient difference sets $D_p,D_q\subseteq S$, realized by directed base paths of lengths $p$ and $q$, respectively, such that
	\[
	M_q\subseteq D_p,\qquad |D_p|=p,
	\]
	and
	\[
	M_p\subseteq D_q,\qquad |D_q|=q.
	\]
\end{proposition}

We prove the proposition in several steps.
We first choose the difference sequences modulo $p$ and $q$, then construct $D_p$ using a bipartite matching, and finally construct $D_q$ while avoiding the differences already used in $D_p$.

\subsection{The difference sequences modulo \texorpdfstring{$p$ and $q$}{p and q}}

The difference sequence modulo $p$ determines how many elements of $D_p$ are chosen from each residue class.
Since $M_q\subseteq D_p$, the residue modulo $p$ of each difference in $M_q$ must occur in the difference sequence used for $D_p$.
Similarly, the residue modulo $q$ of each difference in $M_p$ must occur in the difference sequence used for $D_q$.

Define
\[
A_p=\left\{[kq]_p:1\le k\le\frac{p-1}{2}\right\}\subseteq\mathbb Z_p^*
\]
and
\[
A_q=\left\{[kp]_q:1\le k\le\frac{q-1}{2}\right\}\subseteq\mathbb Z_q^*.
\]
Thus $A_p$ and $A_q$ are precisely the residues of the differences in $M_q$ and $M_p$, respectively.

\begin{lemma}\label{cs}
	There exist Hamiltonian difference sequences
	\[
	\boldsymbol{\rho}^{(p)}
	=
	\bigl(\rho_1^{(p)},\ldots,\rho_p^{(p)}\bigr)
	\]
	and
	\[
	\boldsymbol{\rho}^{(q)}
	=
	\bigl(\rho_1^{(q)},\ldots,\rho_q^{(q)}\bigr)
	\]
	whose unique zero-multiplicity classes $x_p$ and $x_q$ satisfy
	$x_p\notin A_p$ and $x_q\notin A_q.$
	In particular, $\mu_a^{(p)}\ge1$ for $a\in A_p$ and $\mu_b^{(q)}\ge1$ for $b\in A_q.$
\end{lemma}

\begin{proof}
	Since $p\nmid q$, multiplication by $q$ permutes $\mathbb Z_p^*$. Therefore,
	\[
	|A_p|=\frac{p-1}{2}<|\mathbb Z_p^*|,
	\]
	and we may choose $x_p\in\mathbb Z_p^*\setminus A_p.$
	Fix a primitive root $g_p$ modulo $p$ and define $y_p=(1-g_p)^{-1}x_p$, $z_p=-g_py_p.$
	Then $x_p=(1-g_p)y_p=y_p+z_p.$
	Lemma~\ref{pfp} gives a Hamiltonian difference sequence $\boldsymbol{\rho}^{(p)}$ whose unique zero-multiplicity class is $x_p$.
	Since $x_p\notin A_p$, we have $\mu_a^{(p)}\ge1$ for $a\in A_p.$
	
    Similarly, applying the same argument modulo $q$, we obtain a Hamiltonian difference sequence $\boldsymbol{\rho}^{(q)}$ whose unique zero-multiplicity class $x_q$ satisfies
    $x_q\notin A_q.$
    Consequently, $\mu_b^{(q)}\ge 1$ for $b\in A_q.$

\end{proof}

\subsection{Residue-class estimates modulo \texorpdfstring{$p$}{p}}

Put $m=\frac{pq-1}{2},$ so that $S=\{1,2,\ldots,m\}$.
For each $a\in\mathbb Z_p^*$, define
\[
\Lambda_a^{(p)}=\{d\in S:[d]_p=a\}
\]
and
\[
U_a^{(p)}=\Lambda_a^{(p)}\cap\mathbb Z_{pq}^{\times}.
\]
Thus $\Lambda_a^{(p)}$ consists of the differences in $S$ that are congruent to $a$ modulo $p$, while $U_a^{(p)}$ consists of the unit differences among them.

\begin{lemma}
	\label{pre}
	For every $a\in\mathbb Z_p^*$, the following statements hold.
	
	\begin{enumerate}
		\item[(i)]
		\[
		\left|\Lambda_a^{(p)}\right|
		=
		\begin{cases}
			\dfrac{q+1}{2},
			&1\le a\le\dfrac{p-1}{2},\\[6pt]
			\dfrac{q-1}{2},
			&\dfrac{p+1}{2}\le a\le p-1.
		\end{cases}
		\]
		
		\item[(ii)] The set $\Lambda_a^{(p)}$ contains exactly one element
		of $M_q$ when $a\in A_p$, and no element of $M_q$ when
		$a\notin A_p$. Consequently,
		\[
		\left|U_a^{(p)}\right|
		=
		\begin{cases}
			\left|\Lambda_a^{(p)}\right|-1,
			& a\in A_p,\\[4pt]
			\left|\Lambda_a^{(p)}\right|,
			& a\notin A_p,
		\end{cases}
		\]
		and hence
		\[
		\left|U_a^{(p)}\right|
		\ge
		\frac{q-3}{2}.
		\]
		
		\item[(iii)] Distinct elements of $U_a^{(p)}$ have distinct
		residues modulo $q$.
	\end{enumerate}
\end{lemma}

\begin{proof}
    The proofs of part~(i) and~(ii) are straightforward.
	For part~(iii), suppose that $d,d'\in U_a^{(p)}$ have the same residue modulo $q$.
	They also have the same residue modulo $p$, so both $p$ and $q$ divide $d-d'$. Hence
	$pq\mid d-d'.$
	Since $d,d'\in S$ and $m<pq$, we have $|d-d'|<pq$.
	Therefore, $d=d'$.

\end{proof}

\subsection{A matching construction for \texorpdfstring{$D_p$}{Dp}}
\label{subsec:matching-Dp}

The set $M_q$ already provides one difference in each residue class belonging to $A_p$.
For each $a\in\mathbb Z_p^*$, define
\[
\lambda_a=
\begin{cases}
	\mu_a^{(p)}-1,
	&a\in A_p,\\[4pt]
	\mu_a^{(p)},
	&a\notin A_p.
\end{cases}
\]
Thus $\lambda_a$ is the number of additional unit differences that must be chosen from the residue class $a$ modulo $p$.
By Lemma~\ref{cs}, $\lambda_a\ge0$ for every $a\in\mathbb Z_p^*$.
Moreover,
\[
\sum_{a\in\mathbb Z_p^*}\lambda_a
=
\sum_{a\in\mathbb Z_p^*}\mu_a^{(p)}-|A_p|
=
\frac{p+1}{2}.
\]
Let $u=\frac{p+1}{2}.$ The remaining task is to choose $\lambda_a$ unit differences from each residue class $a$ modulo $p$ so that their residues modulo $q$ are pairwise distinct.

Let $G_p=(L,R,E)$ be the bipartite graph with bipartite sets $L$ and $R$, where
\[
L=\left\{(a,i):a\in\mathbb Z_p^*,\ 1\le i\le\lambda_a\right\}
\quad \text{and} \quad
R=\mathbb Z_q^*.
\]
The vertex $(a,i)$ represents one of the $\lambda_a$ unit differences required from the residue class $a$ modulo $p$.
In particular, $|L|=\sum_{a\in\mathbb Z_p^*}\lambda_a=u.$

The edge set $E$ is defined by
\[E=\left\{\bigl((a,i),b\bigr)\in L\times R:[d]_q=b\text{ for some }d\in U_a^{(p)}\right\}.\]
Thus $(a,i)$ is adjacent to $b$ precisely when there is an available unit difference whose residue is $a$ modulo $p$ and $b$ modulo $q$.

By Lemma~\ref{pre}(iii), the difference associated with an edge is unique.
For each $e=\bigl((a,i),b\bigr)\in E$, denote this difference by $d(e)$.
Hence
\[
[d(e)]_p=a
\qquad\text{and}\qquad
[d(e)]_q=b.
\]

For $X\subseteq L$, let $N_{G_p}(X)$ denote the neighborhood of $X$ in $G_p$, the set of all vertices in $R$ that are adjacent to at least one vertex of $X$.
For a single vertex $v$, we write $N_{G_p}(v)$ instead of $N_{G_p}(\{v\})$.
 If $v=(a,i)$, then
\[
N_{G_p}(v)=\left\{[d]_q:d\in U_a^{(p)}\right\}.
\]
Lemma~\ref{pre}(iii) therefore gives
\begin{equation}
	\label{eq:degree}
	\deg_{G_p}(v)=\left|U_a^{(p)}\right|\ge\frac{q-3}{2}.
\end{equation}

An $L$-saturating matching is a set of disjoint edges covering every vertex in $L$.
 It therefore selects the required number of unit differences from every residue class modulo $p$, with pairwise distinct residues modulo $q$.

A condition clearly necessary for the existence of an $L$-saturating
matching is that every subset of $L$ has enough neighbors in $R$, i.e. that
 \[
 \left|N_{G_p}(X)\right|\ge\left|X\right|\quad \text{for all}\ X\subseteq L.
 \]
This is Hall's condition. In 1935, Philip Hall formulated this condition for systems of distinct representatives and proved that it is also sufficient \citep{hallRepresentativesSubsets1935}.
In its equivalent graph-theoretic form, Hall's marriage theorem states that a finite bipartite graph has an $L$-saturating matching if and only if Hall's condition holds.
The following lemma verifies this condition for $G_p$ and then uses the resulting matching to select the required unit differences.

\begin{lemma}\label{mD}
	There exists a matching $\mathcal M$ in $G_p$ that saturates $L$ and satisfies
\begin{equation}
	\label{eq:matching-sum}
	\sum_{d\in M_q}d
	+
	\sum_{e\in\mathcal M}d(e)
	\not\equiv0\pmod q.
\end{equation}
\end{lemma}

\begin{proof}
	We first verify Hall's condition
	\begin{equation}
		\label{eq:hall}
	 \left|N_{G_p}(X)\right|\ge\left|X\right|\quad \text{for all}\ X\subseteq L.
	\end{equation}
	The condition is immediate when $X=\emptyset$.
	Let $\emptyset\ne X\subseteq L$ and choose $v\in X$. By \eqref{eq:degree},
\begin{equation}
	\label{eq:degree-bound}
	|N_{G_p}(X)|\ge|N_{G_p}(v)|=\deg_{G_p}(v)\ge\frac{q-3}{2}.
\end{equation}
	
	Suppose first that $q\ge p+4$. Then
	\[
	\frac{q-3}{2}\ge\frac{p+1}{2}=u\ge|X|.
	\]
	Thus \eqref{eq:hall} holds.
	
	It remains to consider $q=p+2$.
	In this case, $\frac{q-3}{2}=u-1.$
	It follows from \eqref{eq:degree-bound} that Hall's condition holds whenever
	$X\subsetneq L$, since then $|X|\le u-1$.
	
	We only need to consider $X=L$.
	Since $q\equiv2\pmod p$,
	\[
	A_p=\{2,4,\ldots,p-1\}.
	\]
	As $p\ge7$, both $1$ and $3$ lie in $\mathbb Z_p^*\setminus A_p$.
	Choose $a\in\{1,3\}\setminus\{x_p\}.$
	Then $a\notin A_p$ and $a\ne x_p$, so
	\[
	\lambda_a=\mu_a^{(p)}\ge1.
	\]
	Consequently, $L$ contains a vertex $v=(a,1)$.
	Since $a\notin A_p$ and $1\le a\le(p-1)/2$, Lemma~\ref{pre} gives
	\[
	\deg_{G_p}(v)=\left|\Lambda_a^{(p)}\right|=\frac{q+1}{2}=u+1.
	\]
	Therefore,
	\[
	|N_{G_p}(L)|\ge\deg_{G_p}(v)\ge u=|L|.
	\]
	Hall's condition now holds for every $X\subseteq L$.
	By Hall's theorem, $G_p$ has a matching $\mathcal M$ that saturates $L$.
	
	To ensure that the difference set obtained from the matching satisfies the endpoint condition in Lemma~\ref{es}, its sum must be nonzero modulo $q$.
	We now adjust the matching, if necessary, to obtain this additional property.
	We first show that $L$ contains a vertex of degree at least $u+1$.
	
	If $q\ge p+6$, then every vertex of $L$ has degree at least $\frac{q-3}{2}\ge u+1.$
	
	If $q=p+4$, choose $a\in\mathbb Z_p^*\setminus\bigl(A_p\cup\{x_p\}\bigr).$
	Such an $a$ exists because $\left|\mathbb Z_p^*\setminus A_p\right|=\frac{p-1}{2}\ge3.$
	Since $a\ne x_p$, we have $\lambda_a=\mu_a^{(p)}\ge1$, and hence
	$v=(a,1)\in L$.
	Moreover,
	\[\deg_{G_p}(v)=\left|\Lambda_a^{(p)}\right|\ge\frac{q-1}{2}=u+1.\]

	Finally, if $q=p+2$, the vertex chosen above already satisfies $\deg_{G_p}(v)=u+1$.
	
	Thus, in every case, there is a vertex $v\in L$ such that
	\begin{equation}
		\label{eq:large-degree}
		\deg_{G_p}(v)\ge u+1.
	\end{equation}
	The matching $\mathcal M$ has $u$ edges and therefore uses exactly
	$u$ vertices of $R$. By \eqref{eq:large-degree}, the vertex $v$ has a neighbor $b'$
	that is not used by $\mathcal M$.
	
	Let $e$ be the edge of $\mathcal M$ incident with $v$, and let $b$ be
	its endpoint in $R$. Let $e'$ be the edge joining $v$ to $b'$. Then
	\[
	\mathcal M'
	=
	\bigl(\mathcal M\setminus\{e\}\bigr)\cup\{e'\}
	\]
	is also a $L$-saturating matching.
	
	If $\mathcal M$ already satisfies \eqref{eq:matching-sum}, there is nothing to prove.
	Otherwise, using $[d(e)]_q=b$ and $[d(e')]_q=b'$, we obtain
	\[
	\begin{aligned}
		\sum_{d\in M_q}d+\sum_{f\in\mathcal M'}d(f)
		&\equiv d(e')-d(e)\\
		&\equiv b'-b\\
		&\not\equiv0\pmod q,
	\end{aligned}
	\]
	because $b'\ne b$. Thus $\mathcal M'$ satisfies \eqref{eq:matching-sum}.

\end{proof}

\begin{lemma}
	\label{lem:construct-Dp}
	There exists a sufficient difference set $D_p\subseteq S$, realized by a directed base path of length $p$, such that $M_q\subseteq D_p.$
	Moreover, the unit differences in $D_p$ have pairwise distinct residues modulo $q$.
\end{lemma}

\begin{proof}
	Let $\mathcal M$ be the matching given by
	Lemma~\ref{mD}, and define
	\begin{equation}
		\label{eq:Dp-def}
		D_p
		=
		M_q\cup
		\left\{
		d(e):e\in\mathcal M
		\right\}.
	\end{equation}
	Distinct edges of $\mathcal M$ have distinct endpoints in $R$.
	Therefore, their associated differences have pairwise distinct
	residues modulo $q$ and, in particular, are distinct. They are units,
	whereas the elements of $M_q$ are nonunits, so the union in \eqref{eq:Dp-def} is
	disjoint.
	
	Since $\mathcal M$ saturates all $\lambda_a$ vertices $(a, i)$ of $L$, it selects exactly $\lambda_a$ unit differences congruent to $a$ modulo $p$.
	The set $M_q$ contributes one additional difference when $a\in A_p$ and contributes no difference otherwise.
	Hence
	\begin{equation}
		\label{eq:Dp-residue-count}
		\left|\left\{d\in D_p:[d]_p=a\right\}\right|=\mu_a^{(p)}\qquad(a\in\mathbb Z_p^*).
	\end{equation}
	Consequently,
	\[
	|D_p|=\sum_{a\in\mathbb Z_p^*}\mu_a^{(p)}=p
	\]
	and
	\begin{equation}
		\label{eq:Dp-sum-p}
		\sum_{d\in D_p}d
		\equiv
		\sum_{a\in\mathbb Z_p^*}\mu_a^{(p)}a
		\equiv0
		\pmod p.
	\end{equation}
	Lemma~\ref{mD} also gives
	\begin{equation}
		\label{eq:Dp-sum-q}
		\sum_{d\in D_p}d\not\equiv0\pmod q.
	\end{equation}
	
	Recall that
	\[
	\boldsymbol{\rho}^{(p)}
	=
	\bigl(\rho_1^{(p)},\ldots,\rho_p^{(p)}\bigr)
	\]
	is the Hamiltonian difference sequence chosen in
	Lemma~\ref{cs}. For every $a\in\mathbb Z_p^*$, the residue $a$
	occurs exactly $\mu_a^{(p)}$ times in this sequence. By
	\eqref{eq:Dp-residue-count},
	$D_p$ contains exactly the same number of differences congruent to
	$a$ modulo $p$. Therefore, the elements of $D_p$ can be ordered as
	$d_1,\ldots,d_p$ so that
	
	\begin{equation}
		\label{eq:Dp-order}
		[d_i]_p=\rho_i^{(p)}
		\qquad
		(1\le i\le p).
	\end{equation}
	
	Define vertices in $\mathbb Z_{pq}$ by
	\[
	v_0=0
	\qquad\text{and}\qquad
	v_i=\sum_{j=1}^{i}d_j
	\quad
	(1\le i\le p).
	\]
	By \eqref{eq:Dp-order}, the residues $[v_0]_p,[v_1]_p,\ldots,[v_p]_p$ are the partial sums of
	$\boldsymbol{\rho}^{(p)}$.
	Since $\boldsymbol{\rho}^{(p)}$ is a Hamiltonian difference sequence, $[v_0]_p,[v_1]_p,\ldots,[v_{p-1}]_p$ are pairwise distinct, and $[v_p]_p=0.$
	In particular, $[v_i]_p\ne0$ for $1\le i\le p-1.$
	It follows from \eqref{eq:Dp-sum-q} that $[v_p]_q\ne0,$ so $v_p\ne v_0$.
	Moreover, $v_p\ne v_i$ for $1\le i\le p-1$, because $[v_p]_p=0$ whereas $[v_i]_p\ne0$.
	Thus $v_0,v_1,\ldots,v_p$ are pairwise distinct.
	
	Since $[v_0]_p=0$, it follows that
	\[
	\left\{
	[v_1]_p,\ldots,[v_{p-1}]_p
	\right\}
	=
	\mathbb Z_p^*.
	\]
	Since $d_1,\ldots,d_p$ are distinct elements of $S$, the sequence $P_p=(v_0,v_1,\ldots,v_p)$ is a directed path in $T_{p,q}$.
	Finally, \eqref{eq:Dp-sum-p} and \eqref{eq:Dp-sum-q} imply $p\mid v_p$ and $q\nmid v_p.$
	Therefore, $\gcd(v_p,pq)=p.$
	Hence $P_p$ is a directed base path of length $p$, and $\Delta(P_p)=D_p.$
	Consequently $D_p$ is a sufficient difference set of type $p$.

\end{proof}

The unit differences in $D_p$ have pairwise distinct residues modulo $q$.
Therefore, deleting $D_p$ removes at most one unit difference from each nonzero residue class modulo $q$, which is the property needed in the construction of $D_q$.

\subsection{Construction of \texorpdfstring{$D_q$}{Dq} disjoint from
\texorpdfstring{$D_p$}{Dp}}
\label{subsec:construct-Dq}

Fix the sufficient difference set $D_p$ given by Lemma~\ref{lem:construct-Dp}.
Its unit differences have pairwise distinct residues modulo $q$.
This allows the differences for $D_q$ to be chosen independently from the nonzero residue classes modulo $q$.

For each $b\in\mathbb Z_q^*$, define
\[
\Lambda_b^{(q)}
=
\left\{
d\in S:[d]_q=b
\right\}
\]
and
\[
V_b^{(q)}
=
\left(
\Lambda_b^{(q)}\cap\mathbb Z_{pq}^{\times}
\right)
\setminus D_p.
\]
Thus $V_b^{(q)}$ is the set of unit differences in the residue class
$b$ modulo $q$ that remain available after $D_p$ has been chosen.

The same counting argument as in Lemma~\ref{pre}, now applied modulo
$q$, gives the corresponding residue-class sizes. The only additional
point is that the differences already used in $D_p$ must be removed.
By Lemma~\ref{fp}, $D_p$ contains no element of $M_p$, while
Lemma~\ref{lem:construct-Dp} shows that its unit differences have
pairwise distinct residues modulo $q$. Thus removing $D_p$ deletes at
most one unit difference from each nonzero residue class modulo $q$.
These observations give the following analogue of Lemma~\ref{pre}; the proof details are omitted.

\begin{lemma}\label{qre}
	For every $b\in\mathbb Z_q^*$, the following statements hold.
	
	\begin{enumerate}
		\item[(i)]
		\[
			\left|\Lambda_b^{(q)}\right|
			=
			\begin{cases}
				\dfrac{p+1}{2},
				&1\le b\le\dfrac{q-1}{2},\\[6pt]
				\dfrac{p-1}{2},
				&\dfrac{q+1}{2}\le b\le q-1.
			\end{cases}
		\]
		
		\item[(ii)] The set $\Lambda_b^{(q)}$ contains exactly one element
		of $M_p$ when $b\in A_q$, and no element of $M_p$ when
		$b\notin A_q$. Moreover,
		\[
			\left|V_b^{(q)}\right|
			\ge
			\begin{cases}
				\dfrac{p-5}{2},
				&b\in A_q,\\[6pt]
				\dfrac{p-3}{2},
				&b\notin A_q.
			\end{cases}
		\]
		
		\item[(iii)] Distinct elements of $V_b^{(q)}$ have distinct
		residues modulo $p$.
	\end{enumerate}
\end{lemma}

Since $p\ge7$, the bounds in Lemma~\ref{qre}(ii) give
\[
\left|V_b^{(q)}\right|\ge1
\qquad
(b\in A_q)
\]
and
\[
\left|V_b^{(q)}\right|\ge2
\qquad
(b\notin A_q).
\]
These bounds are sufficient to construct $D_q$.

\begin{lemma}
	\label{lem:construct-Dq}
	There exists a sufficient difference set
	\[
	D_q\subseteq S\setminus D_p
	\]
	of type $q$ such that
	\[
	M_p\subseteq D_q.
	\]
\end{lemma}

\begin{proof}
	We first construct a set $D_q^0$. For each $b\in A_q$, include the
	unique element of $M_p$ congruent to $b$ modulo $q$. Since
	\[
	1\le\mu_b^{(q)}\le2
	\qquad
	(b\in A_q),
	\]
	choose an additional $\mu_b^{(q)}-1$ elements from $V_b^{(q)}$.
	
	For each $b\notin A_q$, choose $\mu_b^{(q)}$ elements from
	$V_b^{(q)}$.
	This is possible because $0\le\mu_b^{(q)}\le2$ and $\left|V_b^{(q)}\right|\ge2.$
	The choices made for distinct values of $b$ lie in distinct residue
	classes modulo $q$. They therefore produce a set
	\[
	M_p\subseteq D_q^0\subseteq S\setminus D_p
	\]
	satisfying
	\begin{equation}
		\label{eq:Dq-count}
		\left|
		\left\{
		d\in D_q^0:[d]_q=b
		\right\}
		\right|
		=
		\mu_b^{(q)}
		\qquad
		(b\in\mathbb Z_q^*).
\end{equation}
	Consequently,
	\begin{equation}
		\label{eq:Dq-sum-q}
		\begin{aligned}
			|D_q^0|
			&=
			\sum_{b\in\mathbb Z_q^*}\mu_b^{(q)}
			=
			q,\\
			\sum_{d\in D_q^0}d
			&\equiv
			\sum_{b\in\mathbb Z_q^*}\mu_b^{(q)}b
			\equiv0\pmod q.
		\end{aligned}
	\end{equation}
	
	It remains to ensure that the sum is nonzero modulo $p$. Since $\left|\mathbb Z_q^*\setminus A_q\right|=\frac{q-1}{2}\ge5,$
	and at most three nonzero residue classes have multiplicity different from $1$, there exists
	$b_0\in\mathbb Z_q^*\setminus A_q$ such that $\mu_{b_0}^{(q)}=1.$
	By Lemma~\ref{qre}(ii),
	\[
	\left|V_{b_0}^{(q)}\right|
	\ge
	\frac{p-3}{2}
	\ge2.
	\]
	
	The construction of $D_q^0$ selects exactly one element $d\in V_{b_0}^{(q)}$.
	Choose another element $d'\in V_{b_0}^{(q)}\setminus\{d\}.$
	By Lemma~\ref{qre}(iii), $[d']_p\ne[d]_p.$
	
	If $\sum_{e\in D_q^0}e\not\equiv0\pmod p,$ set $D_q=D_q^0$.
	Otherwise, define
	\[
	D_q
	=
	\left(D_q^0\setminus\{d\}\right)\cup\{d'\}.
	\]
	In the latter case,
	\[
	\sum_{e\in D_q}e
	\equiv
	d'-d
	\not\equiv0\pmod p.
	\]
	Since $d$ and $d'$ have the same residue modulo $q$, the replacement
	does not change the residue counts in
	\eqref{eq:Dq-count} or the congruence modulo $q$ in
	\eqref{eq:Dq-sum-q}. Thus, in either case,
	\begin{equation}
		\label{eq:Dq-sums}
		\sum_{d\in D_q}d\equiv0\pmod q
		\qquad\text{and}\qquad
		\sum_{d\in D_q}d\not\equiv0\pmod p.
	\end{equation}
	Moreover, $M_p\subseteq D_q\subseteq S\setminus D_p$ and $|D_q|=q$.
	
	Recall that
	\[
	\boldsymbol{\rho}^{(q)}
	=
	\bigl(\rho_1^{(q)},\ldots,\rho_q^{(q)}\bigr)
	\]
	is the Hamiltonian difference sequence chosen in
	Lemma~\ref{cs}. By \eqref{eq:Dq-count}, which remains valid
	for $D_q$, its elements can be ordered as $e_1,\ldots,e_q$
	so that
	\begin{equation}
		\label{eq:Dq-order}
		[e_i]_q=\rho_i^{(q)}
		\qquad
		(1\le i\le q).
	\end{equation}
	
	Define $w_0=0$ and $w_i=\sum_{j=1}^{i}e_j$ for $1\le i\le q.$
	By \eqref{eq:Dq-order}, the residues $[w_0]_q,[w_1]_q,\ldots,[w_{q-1}]_q$
	are pairwise distinct and hence comprise all of $\mathbb Z_q$.
	Therefore,
	\[
	\left\{
	[w_1]_q,\ldots,[w_{q-1}]_q
	\right\}
	=
	\mathbb Z_q^*.
	\]
	The Hamiltonian difference sequence also gives $[w_q]_q=0.$
	By \eqref{eq:Dq-sums}, $[w_q]_p\ne0,$
	so $w_q\ne w_0$. Also, $w_q\ne w_i$ for
	$1\le i\le q-1$, because $[w_q]_q=0$ whereas $[w_i]_q\ne0$.
	Thus $w_0,w_1,\ldots,w_q$ are pairwise distinct.
	
	The differences $e_1,\ldots,e_q$ are distinct elements of $S$.
	Hence
	\[
	P_q=(w_0,w_1,\ldots,w_q)
	\]
	is a directed path in $T_{p,q}$.
	Finally, \eqref{eq:Dq-sums} gives $q\mid w_q$ and $p\nmid w_q,$ and therefore $\gcd(w_q,pq)=q.$
	It follows that $P_q$ is a directed base path of length $q$ with $\Delta(P_q)=D_q.$
	Thus $D_q$ is a sufficient difference set of type $q$.

\end{proof}

Lemmas~\ref{lem:construct-Dp} and~\ref{lem:construct-Dq} establish
Proposition~\ref{prop:long-difference-sets}.

\section{Proof of the main theorem}
\label{sec:main-proof}

\begin{proof}[Proof of Theorem~\ref{thm:main}]
	By Proposition~\ref{prop:long-difference-sets}, choose disjoint
	sufficient difference sets $D_p,D_q\subseteq S$. Define
	$R=S\setminus(D_p\cup D_q)$.
	Since $M_p\cup M_q$ contains all nonunit differences in $S$, every
	element of $R$ is a unit of $\mathbb Z_{pq}$. Therefore, for each
	$u\in R$, the singleton $\{u\}$ is a sufficient difference set of
	type $1$, realized by the directed base path $P_u=(0,u).$
	
	It follows that
	\[
	\mathcal D
	=
	\{D_p,D_q\}
	\cup
	\left\{
	\{u\}:u\in R
	\right\}
	\]
	is a partition of $S$ into sufficient difference sets. Hence
	$\mathcal D$ is a complete system of sufficient difference sets for
	$S$. By Corollary~\ref{css}, the corresponding cyclic developments
	form a cyclic Hamilton cycle decomposition of $T_{p,q}$.

\end{proof}

\section{Concluding remarks}

The proof separates the two main sources of difficulty in
Problem~\ref{prob:carousel}. The prime-field sequences control the
multiplicity and order of the residue classes along the two long base
paths, while the matching and replacement arguments lift this local
information to $\mathbb Z_{pq}$ and enforce the endpoint conditions.

The restriction $p\ge7$ is used only to guarantee enough remaining unit
differences in the construction of $D_q$. The remaining products $pq$
with $p\in\{3,5\}$, and more generally all odd composite orders, remain
open.


\section{Acknowledgements}
This work was supported by the National Natural Science
Foundation of China(12301449), and the Scientific Innovation Project for Young Scientists in Shandong Provincial Universities(2024KJG014), the Natural Science Foundation of Henan Province(252300421489), the Young Backbone Teachers of Higher Education Institutions in Henan Province(2024GGJS045), the Key Research Project of Higher Education Institutions in Henan Province(25B110015).

\bibliographystyle{unsrtnat}

\end{document}